\documentclass[11pt]{amsart}

\usepackage{latexsym,rawfonts}
\usepackage{amsfonts,amssymb}
\usepackage{amsmath,amsthm}

\usepackage[plainpages=false]{hyperref}
\usepackage{graphicx}
\usepackage{color}
\usepackage[table]{xcolor}
\usepackage{longtable}

\newcommand{\R}{\mathbb{R}}

\newcommand{\ko}{\mathcal K_0^n}
\newcommand{\koo}{\mathcal K^n}
\newcommand{\ke}{\mathcal K_e^n}
\newcommand{\rnnn}{\mathbb R^{n}}

\newcommand{\rt}{\mathbb R^{n-1}}
\newcommand{\sn}{ {\mathbb{S}^{n-1}}}
\newcommand{\rn}{\mathbb R}

\newcommand{\psum}{{+_{\negthinspace\kern-2pt p}}\,}
\newcommand{\qsum}[1]{{+_{\negthinspace\kern-2pt #1}}\,}
\newcommand{\dpsum}{{\tilde+_{\negthinspace\kern-1pt p}}\,}
\newcommand{\dqsum}[1]{{\tilde+_{\negthinspace\kern-1pt #1}}\,}
\newcommand{\lsub}[1]{\hskip -1.5pt\lower.5ex\hbox{$_{#1}$}}

\numberwithin{equation}{section}

\newtheorem{theo}{Theorem}[section]
\newtheorem{coro}[theo]{Corollary}
\newtheorem{lem}[theo]{Lemma}

\theoremstyle{definition}

\begin{document}

\title{The torsion log-Minkowski problem}

\author[J. Hu]{Jinrong Hu}
\address{School of Mathematics, Hunan University, Changsha, 410082, Hunan Province, China}
\email{hujinrong@hnu.edu.cn}

\begin{abstract}
In this paper, we deal with the torsion log-Minkowski problem without symmetry assumptions via an approximation argument.
\end{abstract}

\maketitle

\baselineskip18pt

\parskip3pt

\section{Introduction}
\label{Sec1}
In the framework of the Brunn-Minkowski theory of convex bodies, the classic Minkowski problem is of central importance. It was introduced and solved by Minkowski himself in \cite{M897, M903} and  there are many excellent subsequent works
(see, e.g., ~\cite{A39,A42,FJ38,B87}). With the development of Minkowski problem, the $L_p$ Minkowski problem was first posed and solved by Lutwak ~\cite{L93} as an analogue of the classic Minkowski problem within the $L_p$ Brunn-Minkowski theory. After that, the $L_{p}$ Minkowski problem has been the breeding ground equipped with many valuable results in a series of paper ~\cite{B17,CL17,B19,F62,L04,Zhu14,Zhu15,Zh15}. In the critical index $p=0$, the $L_{p}$ Minkowski problem is reduced to the logarithmic Minkowski problem prescribing cone-volume measure, which was first solved by B\"{o}r\"{o}czky-Lutwak-Yang-Zhang \cite{BLYZ12} for symmetric convex bodies.  Later, Zhu \cite{Zhu14} solved the log-Minkowski problem for polytopes, and Chen-Li-Zhu \cite{CL19} attacked the general case. There are some other important researches on extensions and analogues of the $L_p$ Minkowski problems for other Borel measures associated with the boundary-value problem (for example, capacity, torsional rigidity), see, e.g.,
\cite{J96,J962,CF10,C15,FZH20,H181,Xiong19,XX22} and their references.

In this paper, we focus on solving the  log-Minkowski problem for torsional rigidity. Recall that the torsional rigidity $T(\Omega)$ of a convex body $\Omega$ in the $n$-dimensional Euclidean space ${\rnnn}$, is defined as
\begin{equation*}\label{tordef}
\frac{1}{T(\Omega)}=\inf\left\{\frac{\int_{\Omega}|\nabla u|^{2}{d}x}{(\int_{\Omega}|u|{d}x)^{2}}: \ u\in W^{1,2}_{0}(\Omega),\ \int_{\Omega}|u|{d}x> 0\right\},
\end{equation*}
where $W^{1,2}_{0}(\Omega)$ denotes the Sobolev space of functions having compact support in $W^{1,2}(\Omega)$, while $W^{1,2}(\Omega)$ represents the Sobolev space of functions having weak derivatives up to first order in $L^{2}(\Omega)$. From the perspective of physics, in the circumstance that $\Omega\subset \rt$ is the cross section of a cylindrical rod $\Omega \times \rn$ under torsion, the torsion rigidity of $\Omega$ is the torque required for unit angle of twist per unit length.

In the background of analysis, the torsional functional can be given by the solution of an elliptic boundary-value problem (see, e.g., ~\cite{CA05}). To be more specific, let $u$ be the unique solution of
\begin{equation}\label{torlapu}
\left\{
\begin{array}{lr}
\Delta u= -2, & x\in \Omega, \\
u=0,  & x\in  \partial \Omega.
\end{array}\right.
\end{equation}
Then
\begin{equation}\label{tordef2}
T(\Omega)=\int_{\Omega}|\nabla u|^{2}{d}x=2\int_{\Omega}udx.
\end{equation}

If the boundary $\partial\Omega$ is of class $C^{2}$, employing the standard regularity results of elliptic equations (see, e.g., Gilbarg-Trudinger \cite{GT01}), we see that $\nabla u$ can be suitably defined $\mathcal{H}^{n-1}$ a.e. on $\partial\Omega$, and $T(\Omega)$ can be expressed as (see, e.g., Proposition 18 of ~\cite{CA05}),
\begin{equation}
\begin{split}\label{tordef3}
T(\Omega)&=\frac{1}{n+2}\int_{\partial\Omega}h(\Omega,g_{\Omega}(x))|\nabla u(x)|^{2}{d}\mathcal{H}^{n-1}(x)\\
&=\frac{1}{n+2}\int_{\sn}h(\Omega,v)|\nabla u(g^{-1}_{\Omega}(v))|^{2}{d}S(\Omega,v),
\end{split}
\end{equation}
where $S(\Omega,\cdot)$ is the surface area measure of $\Omega$.\cite[Theorem 3.1]{CF10} indicated that $~\eqref{tordef3}$ also holds for any convex body in ${\rnnn}$ and established the Hadamard variational formula of $T(\cdot)$, given as (see \cite[Theorem 4.1]{CF10}),
\begin{equation}\label{torhadma}
\frac{d}{dt}T(\Omega+t\Omega_{1})\Big|_{t=0}=\int_{\sn}h(\Omega_{1},v){d}\mu^{tor}(\Omega,v),
\end{equation}
 where the torsion measure $\mu^{tor}(\Omega,\eta)$ is defined by
\begin{equation}\label{tormes2}
\mu^{tor}(\Omega,\eta)=\int_{g^{-1}_{\Omega}(\eta)}|\nabla u(x)|^{2}{d}\mathcal{H}^{n-1}(x)=\int_{\eta}|\nabla u(g^{-1}_{\Omega}(v))|^{2}{d}S(\Omega,v)
\end{equation}
for every Borel subset $\eta\subset {\sn}$.

Proposition 2.5 of ~\cite{CF10} revealed that $\nabla u$ has finite non-tangential limits $\mathcal{H}^{n-1}$ a.e. on $\partial \Omega$ by extending the estimates of harmonic functions proved by Dahlberg \cite{D77} and $|\nabla u|\in L^{2}(\partial \Omega, \mathcal{H}^{n-1})$ without the assumption of smoothness, which illustrate that $~\eqref{tormes2}$ is well-defined a.e. on the unit sphere ${\sn}$, not limited to the case of smoothness, and can be regarded as a Borel measure.

In \cite{CF10}, the Minkowski problem for torsional rigidity was first posed: If $\mu$ is a finite Borel measure on ${\sn}$, what are necessary and sufficient conditions on $\mu$ such that $\mu$ is the torsion measure $\mu^{tor}(\Omega,\cdot)$ of a convex body $\Omega$ in ${\rnnn}$? Colesanti-Fimiani \cite{CF10} proved the existence and uniqueness up to translations of the solution via the variational method, which was firstly proposed by Aleksandrov ~\cite{A39, A42} and was later adopted by Jerison ~\cite{J96}, Colesanti-Nystr\"{o}m-Salani-Xiao-Yang-Zhang ~\cite{C15}.

 Similar to the $L_{p}$ surface measure (see, e.g., ~\cite{S14}),  recently, Chen-Dai \cite{Chen20} defined the $L_{p}$ torsion measure $\mu^{tor}_{p}(\Omega,\cdot)$ by
\begin{equation}\label{pmea1}
\mu^{tor}_{p}(\Omega,\eta)=\int_{g^{-1}_{\Omega}(\eta)}h(\Omega,g_{\Omega}(x))^{1-p}{d}\mu^{tor}(\Omega,\eta)=\int_{\eta}h(\Omega,v)^{1-p}{d}\mu^{tor}(\Omega,\eta)
\end{equation}
for every Borel subset $\eta$ of ${\sn}$. Naturally, the $L_{p}$ Minkowski problem for torsional rigidity first introduced by the authors \cite{Chen20} arises as: For $p\in {\R}$, given a finite Borel measure $\mu$ on ${\sn}$, what are the necessary and sufficient conditions on $\mu$ such that $\mu$ is the $L_{p}$ torsion measure $\mu^{tor}_{p}(\Omega,\cdot)$ of a convex body $\Omega$ in ${\rnnn}$? In \cite{Chen20}, the authors proved the existence and uniqueness of the solution when $p>1$. For the case $0<p<1$, Hu-Liu \cite{HL21} gave a sufficient condition. To our knowledge, there are few results on the existence of solution to the torsion log-Minkowski problem without symmetry assumptions in the case $p=0$. We will enrich this topic in this paper.

We first reveal the cone torsion measure $G^{tor}(\Omega,\cdot)$ of a convex body $\Omega$ via applying \eqref{torhadma}, which is defined for a Borel set $\eta\subset \sn$ by
\begin{equation}\label{CTC}
G^{tor}(\Omega,\eta)=\frac{1}{n+2}\int_{x\in g^{-1}_{\Omega}(\eta)}h(\Omega,g_{\Omega}(x))d\mu^{tor}(\Omega,\eta).
\end{equation}

The Minkowski problem prescribing cone-torsion measure is stated as:

{\bf The Torsion Log-Minkowski Problem.} If $\mu$ is a finite Borel measure on $\sn$, what are necessary and sufficient conditions on $\mu$ to guarantee the existence of a convex body $\Omega\subset \rnnn$ containing the origin that solves the equation
\begin{equation}\label{Gq}
G^{tor}(\Omega,\cdot)=\mu?.
\end{equation}

It was shown in \cite{LU23} with additional assumption that $\mu$ is even, and satisfies the following \emph{subspace mass inequality}
\begin{equation}\label{SB}
\frac{\mu(\xi_{k}\cap \sn)}{|\mu|}<\frac{k}{n}
\end{equation}
for each $k$-dimensional subspace $\xi_{k}\subset \rnnn$ and each $k=1,\ldots,n-1$. Then there exists an origin-symmetric convex body $\Omega$ in $\rnnn$ such that \eqref{Gq} holds.

In this paper, we shall show that the symmetric assumption can be removed.
We first deal with the polytopal case when the given normal vectors are in \emph{general position}.

\begin{theo}\label{main21}
Let $\mu$ be a discrete measure on $\rnnn$ whose support set is not contained in any closed hemisphere and is in general position in dimension $n$. Then there exists a polytope $P$ containing the origin in its interior such that
\[
G^{tor}(P,\cdot)=\mu.
\]

\end{theo}

By an approximation scheme, we shall give the sufficient condition to solve the general torsion log-Minkowski problem without symmetry assumptions.

\begin{theo}\label{main22}
If $\mu$ is a finite, non-zero Borel measure on $\sn$ satisfying the subspace mass inequality \eqref{SB}. Then there exists a convex body $\Omega$ in $\rnnn$ with $o\in \Omega$ such that
\[
G^{tor}(\Omega,\cdot)=\mu.
\]

\end{theo}

We remark that the uniqueness of the torsion log-Minkowski problem is challenging. The key is to discover a logarithmic Brunn-Minkowski inequality for torsional rigidity, which is resemble to the classical logarithmic Brunn-Minkowski inequality first proved by B\"{o}r\"{o}czky-Lutwak-Yang-Zhang \cite{BLZ12} provided that two plane convex bodies are origin-symmetric. Very recently, Crasta-Fragal\`a \cite{CF23} gave some partial answers on this topic, they illustrated that if such a measure is absolutely continuous with constant density, the underlying body is a ball.

This paper is organized as follows. In Sec. \ref{Sec2} , we collect some facts about convex bodies, torsional rigidity and torsion measure. In Sec. \ref{Sec3},  the associated extremal problem is introduced. In Sec. \ref{Sec4}, we solve the discrete torsion log-Minkowski problem. In Sec. \ref{Sec5}, the general torsion log-Minkowski problem is studied by an approximation technique.

\section{Backgrounds}
\label{Sec2}
In this section, we list some basic facts about convex bodies and torsional rigidity that we shall use in what follows.
\subsection{Convex bodies}
There are a great quantity of standard and good references regarding convex bodies, for example, refer to Gardner \cite{G06} and Schneider \cite{S14}.

Denote by ${\rnnn}$ the $n$-dimensional Euclidean space, by $o$ the origin of ${\rnnn}$. $\omega_{n}$ is the $n$-dimensional volume of the unit ball $B^{n}$ in $\rnnn$. We will also use the notation $|\mu|$ for the total mass of a measure $\mu$. Let $G_{n,k}$ denote the Grass-mannian of $k$-dimensional subspaces of $\rnnn$. For $x,y\in {\rnnn}$, $x\cdot y$ denotes the standard inner product.  For $x\in{\rnnn}$, denote by $|x|=\sqrt{x\cdot x}$ the Euclidean norm. The origin-centered ball $B$ is denoted by $\{x\in {\rnnn}:|x|\leq 1\}$, its boundary by ${\sn}$.  Denoted by $C({\sn})$ the set of continuous functions defined on the unit sphere ${\sn}$, by $C^{+}({\sn})$ the set of strictly positive functions in $C({\sn})$.  A compact convex set of ${\rnnn}$ with non-empty interior is called by a convex body. The set of all convex bodies in $\rnnn$ is denoted by $\koo$. The set of all convex bodies containing the origin in the interior is denoted by $\ko$, and the set of all origin-symmetric convex bodies by $\ke$.

If $\Omega$ is a compact convex set in ${\rnnn}$, for $x\in{\rnnn}$, the support function of $\Omega$ with respect to $o$ is defined by
\[
h(\Omega,x)=\max\{x\cdot y:y \in \Omega\}.
\]

For compact convex sets $\Omega$ and $L$ in ${\rnnn}$, any real $a_{1},a_{2}\geq 0$, define the Minkowski combination of $a_{1}\Omega+a_{2}L$ in ${\rnnn}$ by
\[
a_{1}\Omega+a_{2}L=\{a_{1}x+a_{2}y:x\in \Omega,\ y\in L\},
\]
and its support function is given by
\[
h({a_{1}\Omega+a_{2}L},\cdot)=a_{1}h(\Omega,\cdot)+a_{2}h(L,\cdot).
\]

For any compact convex set $\Omega$ in ${\rnnn}$, the radial function $\rho(\Omega,u)$ of $\Omega$ with respect to $o$ is expressed as
\[
\rho(\Omega,u)=\max\{\lambda:\lambda u\in \Omega\},\ \forall u\in  {\rnnn} \backslash\{o\}.
\]

The mag $g_{\Omega}:\partial \Omega\rightarrow \sn$ denotes the Gauss map of $\partial \Omega$.

 The Hausdorff metric $\mathcal{D}(\Omega,L)$ between two compact convex sets $\Omega$ and $L$ in ${\rnnn}$, is expressed as
\[
\mathcal{D}(\Omega,L)=\max\{|h(\Omega,v)-h(L,v)|:v\in {\sn}\}.
\]
Let $\Omega_{j}$ be a sequence of compact convex set in ${\rnnn}$, for a compact convex set $\Omega_{0}$ in ${\rnnn}$, if $\mathcal{D}(\Omega_{j},\Omega_{0})\rightarrow 0$, then $\Omega_{j}$ converges to $\Omega_{0}$.

For each $h\in C^{+}(\sn)$, the \emph{Wulff shape} generated by $h$, denoted by $[h]$, is the convex body defined by
\[
[h]=\{x\in \rnnn: x \cdot v\leq h(v), {\rm for \ all} \ v\in \sn\}.
\]

For a compact convex set $\Omega$ in ${\rnnn}$, the diameter of $\Omega$ is denoted by
\[
diam(\Omega)=\max\{|x-y|:x,y \in \Omega\}.
\]

For any convex body $\Omega$ in ${\rnnn}$ and $v\in {\sn}$, the support hyperplane $H(\Omega,v)$ in direction $v$ is defined by
\[
H(\Omega,v)=\{x\in {\rnnn}:x\cdot v=h(\Omega,v)\},
\]
the half-space $H^{-}(\Omega,v)$ in the direction $v$ is defined by
\[
H^{-}(\Omega,v)=\{x\in {\rnnn}:x\cdot v\leq h(\Omega,v)\},
\]
and the support set $F(\Omega,v)$ in the direction $v$ is defined by
\[
F(\Omega,v)=\Omega\cap H(\Omega,v).
\]

Denote by $\mathcal{P}$ the set of polytopes in ${\rnnn}$, suppose that the unit vectors $v_{1},\ldots,v_{N}$ $(N\geq n+1)$ are not concentrated on any closed hemisphere of ${\sn}$, by $\mathcal{P}(v_{1},\ldots,v_{N})$ the set with $P\in \mathcal{P}(v_{1},\ldots,v_{N})$ satisfying
\[
P=\bigcap_{k=1}^{N}H^{-}(P,v_{k}).
\]
It is easy to see that, for $P\in \mathcal{P}(v_{1},\ldots,v_{N})$, then $P$ has at most $N$ facets and the outer normals of $P$ are a subset of $\{v_{1},\ldots,v_{N}\}$.

A special collection of polytopes are those facet normals are in {\emph general position}. We say $v_{1},\ldots, v_{N}$ are in general position in dimension $n$ if for any $n$-tuple $1\leq i_{1}< i_{2}< \ldots<i_{n}\leq N$, the vectors $v_{i_{1}},\ldots,v_{i_{n}}$ are linearly independent.

\subsection{Torsional rigidity and torsion measure}

We first list some properties corresponding to the solution $u$ of \eqref{torlapu}, as follows.

\begin{lem}\cite{K85,CF10}\label{CE}
Let $\Omega$ be a compact convex set of $\rnnn$ and let $u$ be the solution of \eqref{torlapu} in $\Omega$. Then $\sqrt{u}$ is a concave function in $\Omega$.
\end{lem}
Let $M_{\Omega}=\max_{\overline{\Omega}}u$, for every $t\in [0,M_{\Omega}]$, we define
\[
\Omega_{t}=\{x\in \Omega|\ u(x)>t\}.
\]
By lemma \ref{CE}, we know that $\Omega_{t}$ is convex for every $t$.

\begin{lem}\cite{CF10}\label{argu}
Let $\Omega$  be a compact convex set in $\rnnn$ and let $u$ be the solution of ~\eqref{torlapu} in $\Omega$, then
\begin{equation}\label{grdguji}
|\nabla u(x)|\leq diam(\Omega),\ \forall x\in  \Omega.
\end{equation}
\end{lem}

For any $ x\in\partial \Omega$, $0<b<1$, the non-tangential cone is defined as
\begin{equation*}\label{capm2}
\Gamma(x)=\left\{y\in \Omega:dist(y,\partial \Omega)> b|x-y|\right\}.
\end{equation*}

\begin{lem}\cite{CF10,D77}\label{Nonfin}
Let $\Omega$ be a compact convex set in ${\rnnn}$ and let $u$ be the solution of \eqref{torlapu} in $\Omega$. Then the non-tangential limit
\begin{equation*}\label{capm3}
\nabla u(x)=\lim_{y\rightarrow x, \ y \in \Gamma(x)}\nabla u(y),
\end{equation*}
exists for $\mathcal{H}^{n-1}$ almost all $x\in\partial \Omega$. Furthermore, for $\mathcal{H}^{n-1}$ almost all $x\in \partial \Omega$,
\begin{equation*}
\nabla u(x)=-|\nabla u(x)|g_{\Omega}(x)\quad{\rm and}\quad |\nabla u|\in L^{2}(\partial \Omega,  \mathcal{H}^{n-1}).
\end{equation*}
\end{lem}

Let $\{\Omega_{i}\}_{i=0}^{\infty}$ be a sequence of convex bodies in ${\rnnn}$, on the one hand, with respect to torsional rigidity, there are the following properties.
\begin{lem}\cite{CF10}\label{T1}
(i) It is positively homogenous of order $n+2$, i.e.,
\begin{equation*}\label{torhom}
T(m\Omega_{0})=m^{n+2}T(\Omega_{0}),\ m> 0.
\end{equation*}

(ii) It is translation invariant. That is
\begin{equation*}\label{tranin}
T(\Omega_{0} +x_{0})=T(\Omega_{0}),\ \forall x_{0}\in {\rnnn}.
\end{equation*}

(iii) If $\Omega_{i}$ converges to $\Omega_{0}$ in the Hausdorff metric as $i\rightarrow \infty$ (i.e., $\mathcal{D}(\Omega_{i},\Omega_{0})\rightarrow 0$ as $i\rightarrow \infty$), then
\begin{equation*}\label{Thousdo}
\lim_{i\rightarrow \infty}T(\Omega_{i})=T(\Omega_{0}).
\end{equation*}

(iv) It is monotone increasing, i.e., $T(L)\leq T(\Omega)$ if $L\subset \Omega$.
\end{lem}

Recall that torsional rigidity satisfies the following isoperimetric-type inequality:
\begin{lem}\cite[de Saint Venant inequality]{PS51} \label{JKT} Let $\Omega$ be a convex body in $\rnnn$, then de Saint Venant inequality is
\begin{equation}\label{IS2}
\left(\frac{T(\Omega)}{T(B^{n})} \right)^{\frac{1}{n+2}}\leq \left( \frac{|\Omega|}{|B^{n}|} \right)^{\frac{1}{n}}.
\end{equation}
More precisely, since $T(B^{n})=\frac{\omega_{n}}{n(n+2)}$, \eqref{IS2} becomes
\begin{equation}\label{IS}
T(\Omega)\leq \frac{|\Omega|^{\frac{n+2}{n}}}{n(n+2)\omega^\frac{2}{n}_{n}},
\end{equation}
with equality if and only if $\Omega$ is a ball.

\end{lem}

As showed in \cite[Theorem 4.1]{CF10}, the torsion measure stems from the differential of torsional rigidity.
\begin{lem}\label{argu2}Let $\Omega$, $\Omega_{1}$ be convex bodies in ${\rnnn}$ and let $h(\Omega_{1},\cdot) $ be the support function of $\Omega_{1}$. For sufficiently small $|t|> 0$, while $\Omega_{t}$ is the Wulff shape of $h_{t}$ defined by
\[
h_{t}(v)=h(\Omega,v)+th(\Omega_{1},v)+o(t,v),\quad v\in \sn,
\]
where $o(t,\cdot)/t\rightarrow 0$ uniformly on $\sn$ as $t\rightarrow 0$, then
\begin{equation}\label{Thadama66}
\frac{d}{dt}T(\Omega_{t})\big|_{t=0}=\int_{{\sn}}h(\Omega_{1},v){d}\mu^{tor}(\Omega,v).
\end{equation}
\end{lem}

On the other hand, with regard to torsion measure, there are also the following properties.

\begin{lem}\cite{CF10}\label{meau1}
(a) It is positively homogenous of order $n+1$, i.e.,
\begin{equation*}\label{Tmeahom}
\mu^{tor}(m\Omega_{0},\cdot)=m^{n+1}\mu^{tor}(\Omega_{0},\cdot),\ m> 0.
\end{equation*}

(b) It is translation invariant. That is
\begin{equation*}\label{Tmeatranin}
\mu^{tor}(\Omega_{0} +x_{0})=\mu^{tor}(\Omega_{0}),\ \forall x_{0}\in {\rnnn}.
\end{equation*}

(c) It is absolutely continuous with respect to the surface area measure.

(d) For any fixed $i\in\{0,1,\ldots\}$, if $\Omega_{i}$ converges to $\Omega_{0}$ in the Hausdorff metric as $i\rightarrow \infty$ (i.e., $\mathcal{D}(\Omega_{i},\Omega_{0})\rightarrow 0$ as $i\rightarrow \infty$), then the sequence of $\mu^{tor}(\Omega_{i},\cdot)$ converges weakly in the sense of measures to $\mu^{tor}(\Omega_{0},\cdot)$ as $i\rightarrow \infty$ .
\end{lem}

\subsection{Cone-torsion measure}
Analogous to torsion measure \eqref{Thadama66}, the differential of torsional rigidity produces the cone-torsion measure.
\begin{lem}\label{Con}
Let $\Omega$, $\Omega_{1}$ be convex bodies in ${\rnnn}$ and let $h(\Omega_{1},\cdot) $ be the support function of $\Omega_{1}$. For sufficiently small $|t|> 0$, while $\Omega_{t}$ is the Wulff shape of $h_{t}$ defined by
\[
\log h_{t}(v)=
\log h(\Omega,v)+th(\Omega_{1},v)+o(t,v),\quad v\in \sn,
\]
where $o(t,\cdot)/t\rightarrow 0$ uniformly on $\sn$ as $t\rightarrow 0$, then
\begin{equation}\label{CO}
\frac{d}{dt}T(\Omega_{t})\big|_{t=0}=(n+2)\int_{{\sn}}h(\Omega_{1},v){d}G^{tor}(\Omega,v).
\end{equation}
\end{lem}
We remark that the cone torsion measure $G^{tor}(\Omega,\cdot)$ is homogeneous of degree $(n+2)$, and $T(\Omega)=G^{tor}(\Omega,\sn)$.

\section{The associated extremal problem}

\label{Sec3}

In this section, we study an extremal problem associated to the discrete torsion log-Minkowski problem for torsional rigidity.

Suppose that $\beta_{1},\ldots,\beta_{N}\in (0,\infty)$, the unit vectors $v_{1},\ldots,v_{N}$ $(N\geq n+1)$ are not concentrated on any closed hemisphere, $\mu$ is the discrete measure on ${\sn}$ such that
\begin{equation}\label{bordef}
\mu=\sum_{k=1}^{N}\beta_{k}\delta_{v_{k}}(\cdot),
\end{equation}
where $\delta_{v_{k}}$ denotes the delta measure that is concentrated at the point $v_{k}$.

As is established in ~\cite{Zhu14},  we define the functional $\Phi_{P}:P\rightarrow {\R}$ by
\begin{equation}\label{phimax}
\Phi_{P}(\gamma)=\sum_{k=1}^{N}\beta_{k}\log(h(P,v_{k})-\gamma\cdot v_{k}).
\end{equation}

Now, we consider the following the extreme value problem,
\begin{equation}\label{extre3}
\inf\left\{\max_{\gamma\in Q}\Phi_{Q}(\gamma):Q\in \mathcal{P}(\nu_{1},\ldots,\nu_{N}),T(Q)=1\right\}.
\end{equation}

In the light of \eqref{extre3}, we begin with showing that there exists a unique $\gamma(P)\in Int(P)$ (i.e., the interior of $P$) such that
\begin{equation}\label{extre4}
\Phi_{P}(\gamma(P))=\max_{\gamma\in Int (P)}\Phi_{P}(\gamma),
\end{equation}
with the following results.
\begin{lem}\label{INT}
Let $\beta_{1},\ldots,\beta_{N}\in (0,\infty)$, suppose that the unit vectors $v_{1},\ldots,v_{N}$ $(N\geq n+1)$ are not concentrated on any closed hemisphere, $P\in \mathcal{P}(u_{1},\ldots,u_{N})$, then there exists a unique maximum point $\gamma\in Int(P)$ such that
\[
\Phi_{P}(\gamma(P))=\max_{\gamma\in P}\Phi_{P}(\gamma).
\]
\end{lem}
\begin{proof}
 Firstly, we show the existence and uniqueness of the maximum point. By the aid of the concavity $\log t$ on $[0,\infty)$, thus for $0<\alpha<1$, $\gamma_{1}, \gamma_{2}\in P$, we obtain
\begin{align*}\label{}
&\Phi_{P}(\alpha\gamma_{1}+(1-\alpha)\gamma_{2})\notag\\
&=\sum_{k=1}^{N}\beta_{k}\log [h(P,v_{k})-(\alpha\gamma_{1}+(1-\alpha)\gamma_{2})\cdot v_{k}]\notag\\
&\geq \alpha\sum_{k=1}^{N}\beta_{k}\log(h(P,v_{k})-\gamma_{1}\cdot v_{k})+ (1-\alpha)\sum_{k=1}^{N}\beta_{k}\log(h(P,v_{k})-\gamma_{2}\cdot v_{k})\notag\\
&=\alpha\Phi_{P}(\gamma_{1})+(1-\alpha)\Phi_{P}(\gamma_{2}).
\end{align*}
In the light of the fact that $P$ is convex, $\alpha\gamma_{1}+(1-\alpha)\gamma_{2}\in P$. Then, equality holds if and only if
\[
h(P,v_{k})-\gamma_{1}\cdot v_{k}=h(P,v_{k})-\gamma_{2}\cdot v_{k},\quad k=1,\ldots,N.
\]
Namely,
\[
\gamma_{1}\cdot v_{k}=\gamma_{2}\cdot v_{k},\quad k=1,\ldots,N.
\]
Since the unit vectors $v_{1},\ldots,v_{N}$ $(N\geq n+1)$ are not concentrated on any closed hemisphere, so we get $\gamma_{1}=\gamma_{2}$, which implies that $\Phi_{P}$ is strictly concave. Then, the existence and uniqueness of a maximum point are obtained by the continuity and strictly concavity of $\Phi_{P}$ and the compactness of $P$ .

Secondly, we are prepared to show that
\[
\Phi_{P}(\gamma(P))=\max_{\gamma\in Int(P)}\Phi_{P}(\gamma)
\]
with $\gamma\in Int(P)$. To prove that, argue by contradiction, assume that a sequence of interior points $\gamma_{j}\rightarrow \partial P$, then $\Phi_{P}(\gamma_{j})\rightarrow -
 \infty$, which follows from $\log {0}=-\infty$.

 Via the above statement, we see that the proof is completed.
\end{proof}

Next, we do some preparations.
\begin{lem}\label{Pcvg}
Let $\beta_{1},\ldots,\beta_{N}\in (0,\infty)$, and the unit vectors $v_{1},\ldots,v_{N}$ $(N\geq n+1)$ are not concentrated on any closed hemisphere, suppose $P_{i}\in \mathcal{P}(v_{1},\ldots, v_{N})$ and $P_{i}\rightarrow P$ as $i\rightarrow \infty$. Then
\[
\lim_{i\rightarrow \infty}\gamma(P_{i})=\gamma(P),\quad  \lim_{i\rightarrow \infty}\Phi_{P_{i}}(\gamma (P_{i}))=\Phi_{P}(\gamma (P)).
\]
\begin{proof}
We adopt contradiction arguments along the similar line as  \cite{Zhu14}. Let $P_{i_{j}}$ be a subsequence of $P_{i}$ such that $P_{i_{j}}$ converges to $P$ , satisfying $\gamma(P_{i_{j}})\rightarrow \gamma_{0}$, but $\gamma_{0}\neq \gamma(P)$. Here clearly, $\gamma_{0}\in P$.

 We first show that $\gamma_{0}$ is an interior point of $P$. Since $\gamma(P)\in Int (P)$ and $P_{i}$ converges to $P$, there exists an $N_{0}>0$ such that
\[
h(P_{i},v_{k})-\gamma(P)\cdot v_{k}> c_{0}
\]
for all $k=1,\ldots, N$, $i>N_{0}$, and here $c_{0}=\frac{1}{2}\min_{v\in \sn}\{ h(P,v)-\gamma(P)\cdot v\}>0$. Then, it follows that
\begin{equation}\label{iuq}
\Phi_{P_{i}}(\gamma(P_{i}))\geq\Phi_{P_{i}}(\gamma(P))>\left(\sum^{N}_{k=1}\beta_{k}\right)\log \frac{c_{0}}{2}, \ {\rm for } \ i> N_{0}.
\end{equation}
If $\gamma_{0}$ is  a boundary point of $P$, then $\lim_{j\rightarrow\infty}\Phi_{P_{i_{j}}}(\gamma(P_{i_{j}}))=-\infty$, which  contradicts to \eqref{iuq}.

In view of the fact that $\gamma_{0}$ is an interior point of $P$ with $\gamma_{0}\neq \gamma(P)$,  then by ~\cite[Theorem 1.8.8]{S14} and the definition of $\gamma(P)$ showed in ~\eqref{extre4}, it is clear to imply that
\begin{equation}\label{P1}
\Phi_{P}(\gamma_{0})< \Phi_{P}(\gamma(P)).
\end{equation}
On the other hand, by the continuity of $\Phi_{P}(\gamma(\cdot))$ in $P$ and $ \gamma(\cdot)$, we have
\begin{equation}\label{P2}
\lim_{j\rightarrow \infty}\Phi_{P_{i_{j}}}(\gamma(P_{i_{j}}))=\Phi_{P}(\gamma_{0}).
\end{equation}
Meanwhile,
\begin{equation}\label{P3}
\lim_{j\rightarrow \infty}\Phi_{P_{i_{j}}}(\gamma(P))=\Phi_{P}(\gamma(P)).
\end{equation}
Combining ~\eqref{P1}, ~\eqref{P2}, ~\eqref{P3}, we get
\begin{equation}\label{ctrat}
\lim_{j\rightarrow \infty}\Phi_{P_{i_{j}}}(\gamma(P_{i_{j}}))<\lim_{j\rightarrow \infty}\Phi_{P_{i_{j}}}(\gamma(P)).
\end{equation}
But for any $P_{i_{j}}$, the fact is that
\[
\Phi_{P_{i_{j}}}(\gamma(P_{i_{j}}))\geq \Phi_{P_{i_{j}}}(\gamma(P)).
\]
Therefore,
\[
\lim_{j\rightarrow \infty}\Phi_{P_{i_{j}}}(\gamma(P_{i_{j}}))\geq \lim_{j\rightarrow \infty}\Phi_{P_{i_{j}}}(\gamma(P)).
\]
This contradicts to ~\eqref{ctrat}. So, $\lim_{i\rightarrow \infty}\gamma(P_{i})=\gamma(P)$. Make use of the continuity of $\Phi_{P}(\cdot)$, we get
\[
\lim_{i\rightarrow \infty}\Phi_{P_{i}}(\gamma(P_{i}))=\Phi_{P}(\gamma(P)).
\]
The proof is completed.
\end{proof}
\end{lem}

The following key lemma proved by \cite{GXZ23} showed that the polytopes whose normals are in general position: if the polytope gets large, then it has to get large uniformly in all directions, which is helpful to get uniform priori bounds.
\begin{lem}\label{GGH}
Let $v_{1},\ldots, v_{N}$ be $N$ unit vectors that are not contained in any closed hemisphere and $P_{i}$ be a sequence of polytopes in $P\in \mathcal{P}(v_{1},\ldots,v_{N})$. Assume the vectors $v_{1},\ldots, v_{N}$ are in general position in dimension $n$. If the outer radii of $P_{i}$ are not uniformly bounded in $i$, then their inner radii are not uniformly bounded in $i$ either.
\end{lem}

Using Lemma \ref{GGH}, we obtain the following result.
\begin{coro}\label{GGH2}
Let $v_{1},\ldots, v_{N}\in \sn$ be $N$ unit vectors are not contained in any closed hemisphere and $P\in \mathcal{P}(v_{1},\ldots,v_{N})$. Assume that $v_{1},\ldots,v_{N}$ are in general position in dimension $n$. If the outer radius of $P_{i}$ is not uniformly bounded, then torsional rigidity $T(P_{i})$ is also unbounded.
\end{coro}
\begin{proof}
This holds by using Lemma \ref{GGH}, the homogeneity and the translation invariance of $T$, and the fact that $T(B^{n})$ is positive for the centered unit ball $B^{n}$.
\end{proof}

It's required to prove the following lemmas for giving the existence of the solution to the torsion log-Minkowski problem.

\begin{lem}\label{mexist}
If $\beta_{1},\ldots,\beta_{N}\in (0,\infty)$, the unit vectors $v_{1},\ldots,v_{N}$ $(N\geq n+1)$ are in general position in dimension $n$, then there exists a polytope $P\in \mathcal{P}(v_{1},\ldots,v_{N})$ solving \eqref{extre3} such that $P$ has exactly $N$ facets, $\gamma(P)=o$, $T(P)=1$ and
\[
\Phi_{P}(o)=\inf\left\{\max_{\gamma\in Q}\Phi_{Q}(\gamma):Q\in \mathcal{P}(v_{1},\ldots,v_{N}),T(Q)=1\right\}.
\]
\end{lem}
\begin{proof}
 By virtue of the translation invariance of $\Phi_{P}$, we can choose a sequence $P_{i}\in \mathcal{P}(v_{1},\ldots,v_{N})$ with $\gamma(P_{i})=o$ and $T(P_{i})=1$  is a minimizing sequence of problem \eqref{extre3}.

 Corollary \ref{GGH2} shows that $P_{i}$ is  bounded. This together with the Blaschke selection theorem (see, e.g., ~\cite[Theorem 1.8.7]{S14}), we know that there exists a subsequence of $P_{i}$ that converges to a $P$. Making use of the continuity of $T$, we have $T(P)=1$, combine with de Saint Venant inequality given by lemma \ref{JKT}, we know that $|P|\geq c_{0}>0$, which implies that $P$ is non-degenerate. Now, by Lemma \ref{Pcvg},  there is $\gamma (P)=\lim_{i\rightarrow \infty}\gamma(P_{i})=o$, and
by virtue of the definition of $\Phi_{P}$, there is
\[
\Phi_{P}(o)=\lim_{i\rightarrow\infty}\Phi_{P_{i}}(o)=\inf\left\{\max_{\gamma\in Q}\Phi_{Q}(\gamma):Q\in \mathcal{P}(v_{1},\ldots,v_{N}),T(Q)=1\right\}.
\]

Secondly, we are prepared to prove that $P$ has exactly $N$ facets. By contradiction, there exists $i_{0}\in \{1,\ldots,N\}$ such that
$F(P,u_{i_{0}})=P\cap H(P,u_{i_{0}})$ is not the facet of $P$. By virtue of ~\cite[Section 2.4]{S14} and the same arguments of \cite[Lemma 5.1]{Xiong19}, choose sufficiently small $\delta>0$, let the polytope be
\[
P_{\delta}=P\cap \{ \gamma: \gamma\cdot v_{i_{0}}\leq h(P,v_{i_{0}})-\delta\}\in \mathcal{P}(v_{1},\ldots,v_{N}),
\]
and satisfy
\[
\alpha P_{\delta}=\alpha(\delta)P_{\delta}=T(P_{\delta})^{-\frac{1}{n+2}}P_{\delta}.
\]
Then, $T(\alpha P_{\delta})=1$ and as $\delta\rightarrow 0^{+}$, $\alpha P_{\delta}\rightarrow P$. Moreover, by Lemma \ref{Pcvg}, we get
\[
\gamma_{p}(P_{\delta})\rightarrow \gamma_{p}(P)=o\in Int (P),\ as\ \delta \rightarrow 0^{+}.
\]
Hence, for sufficiently small $\delta> 0$, we take
\[
\gamma (P_{\delta})\in Int(P)
\]
and
\[
h(P,v_{k})> \gamma(P_{\delta})\cdot v_{k}+\delta,\ for \ k\in \{ 1,\ldots, N\}.
\]
Next, we are devoted to showing that
\begin{equation}\label{ctdo}
\Phi_{\alpha P_{\delta}}(\gamma (\alpha P_{\delta}))< \Phi_{P}(\gamma(P))=\Phi(o).
\end{equation}
In view of the fact that $\gamma(\alpha P_{\delta})=\alpha \gamma(P_{\delta})$, it follows that
\begin{align*}\label{}
&\Phi_{\alpha P_{\delta}}(\gamma(\alpha P_{\delta}))\notag\\
&=\sum_{k=1}^{N}\beta_{k}\log(h(\alpha P_{\delta},v_{k})-\gamma(\alpha P_{\delta})\cdot v_{k})\notag\\
&=\log\alpha\sum_{k=1}^{N}\beta_{k}+\sum_{k=1}^{N}\beta_{k}\log(h(P_{\delta},v_{k})-\gamma(P_{\delta})\cdot v_{k})\notag\\
&=\log\alpha\sum_{k=1}^{N}\beta_{k}+\sum_{k=1}^{N}\beta_{k}\log(h(P,v_{k})-\gamma(P_{\delta})\cdot v_{k})-\beta_{i_{0}}\log(h(P,v_{i_{0}})-\gamma(P_{\delta})\cdot v_{i_{0}})\notag\\
&\quad+\beta_{i_{0}}\log(h(P,v_{i_{0}})-\gamma(P_{\delta})\cdot v_{i_{0}}-\delta)\notag\\
&=\Phi_{P}(\gamma(P_{\delta}))+H(\delta),
\end{align*}
where
\begin{align}\label{Heq}
H(\delta)&=\log\alpha\sum_{k=1}^{N}\beta_{k}-\beta_{i_{0}}\log(h(P,v_{i_{0}})-\gamma(P_{\delta})\cdot v_{i_{0}})\notag\\
&\quad+\beta_{i_{0}}\log(h(P,v_{i_{0}})-\gamma(P_{\delta})\cdot v_{i_{0}}-\delta).
\end{align}
Now, we are desired to get $H(\delta)< 0$ to admit ~\eqref{ctdo}. Let  $q_{0}$ be the diameter of $P$, then
\[
0<h(P,u_{i_{0}})-\gamma(P_{\delta})\cdot u_{i_{0}}-\delta<h(P,u_{i_{0}})-\gamma(P_{\delta})\cdot u_{i_{0}}< q_{0}.
\]
Thus,  by the concavity of $\log t$ in $[0,\infty)$, we get
\[
\log(h(P,u_{i_{0}})-\gamma (P_{\delta})\cdot u_{i_{0}}-\delta)-\log(h (P,u_{i_{0}})-\gamma (P_{\delta})\cdot u_{i_{0}})<\log(q_{0}-\delta)-\log q_{0}.
\]
Hence, together with ~\eqref{Heq}, we have
\begin{align*}
H(\delta)<M(\delta),
\end{align*}
where
\[
M(\delta)=-\frac{1}{n+2}\log T(P_{\delta})\left(\sum_{k=1}^{N}\beta_{k}\right)+\beta_{i_{0}}(\log(q_{0}-\delta)-\log q_{0}).\]
Now, by exploiting the Hadamard variational formula of torsional rigidity reflected by Lemma \ref{argu2}, we get
\begin{align}\label{Hlim}
M^{'}(\delta)&=-\frac{1}{n+2}\left(\sum^{N}_{k=1}\beta_{k} \right)\frac{1}{T(P_{\delta})}\frac{d T(P_{\delta})}{d \delta}-\frac{\beta_{i_{0}}}{q_{0}-\delta}\notag\\
&=-\frac{1}{n+2}\left(\sum^{N}_{k=1}\beta_{k} \right)\frac{1}{T(P_{\delta})}\sum^{N}_{k=1}h^{'}(P_{\delta},v_{k})\mu^{tor}(P_{\delta},\{v_{k}\})-\frac{\beta_{i_{0}}}{q_{0}-\delta}.
\end{align}
Suppose $\mu^{tor}(P,\{v_{k}\})\neq 0$ for some $k\in \{1,\ldots, N\}$. By the absolute continuity of $\mu^{tor}(P,\cdot)$ with regard to $S(P,\cdot)$, we know $S(P,\{v_{k}\})\neq 0$. As a result, we conclude  that $P$ has a facet with normal vector $v_{k}$, and by the definition of $P_{\delta}$, for sufficiently small $\delta> 0$, we have $h(P_{\delta},v_{k})=h({P},v_{k})$, this illustrates that $h^{'}(v_{k},0^{+})=0$, where
\[
h^{'}(v_{k},0^{+})=\lim_{\delta\rightarrow 0^{+}}\frac{h(P_{\delta},v_{k})-h({P},v_{k})}{\delta}.
\]
This together with $P_{\delta}\rightarrow P$ as $\delta \rightarrow 0^{+}$, we have
\begin{equation*}\label{ulim}
\sum^{N}_{k=1}h^{'}(P_{\delta},v_{k})\mu^{tor}(P_{\delta},\{v_{k}\})\rightarrow \sum^{N}_{k=1}h^{'}(v_{k},0^{+})\mu^{tor}(P,\{v_{k}\})=0, \ as \ \delta\rightarrow 0^{+},
\end{equation*}
which tells that, for sufficiently small $\delta$,
\begin{equation}\label{alim}
M^{'}(\delta)< 0.
\end{equation}
Since $M(0)=0$, for sufficiently small $\delta> 0$, we know $M(\delta)<0$,  it directly yields $H(\delta)< 0$. So, there exists a $\delta_{0}> 0$ such that $P_{\delta_{0}}\in \mathcal{P}(v_{1},\ldots,v_{N})$ and
\[
\Phi_{\alpha_{0}P_{\delta_{0}}}(\gamma(\alpha_{0}P_{\delta_{0}}))<\Phi_{P}(\gamma(P_{\delta_{0}}))\leq\Phi_{P}(\gamma(P))
=\Phi_{P}(o),
\]
where $\alpha_{0}=T(P_{\delta_{0}})^{-\frac{1}{n+2}}$. Set $P_{0}=\alpha_{0}P_{\delta_{0}}-\gamma(\alpha_{0}P_{\delta_{0}})$, for $P_{0}\in \mathcal{P}(v_{1},\ldots,v_{N})$, then we get
\[
T(P_{0})=1,\ \gamma(P_{0})=o,\ \Phi_{P_{0}}(o)< \Phi_{P}(o).
\]
This is a contradiction. So, $P$ has exactly $N$ facets.
\end{proof}

\section{Dealing with the discrete torsion log-Minkowski problem}
\label{Sec4}

In this section, we first attack the discrete torsion log-Minkowski problem.
\begin{theo}\label{maint2}
 Let $\beta_{1},\ldots,\beta_{N}\in (0,\infty)$, and the unit vectors $v_{1},\ldots,v_{N}$ $(N\geq n+1)$ are in general position in dimension $n$. If there exists a polytope $P\in \mathcal{P}(u_{1},\ldots,u_{N})$ satisfying $\gamma(P)=o$ and $T(P)=1$ such that
\begin{equation*}\label{}
\Phi_{P}(o)=\inf\left\{\max_{\gamma\in Q}\Phi_{Q}(\gamma):Q\in \mathcal{P}(v_{1},\ldots,v_{N}),T(Q)=1\right\}.
\end{equation*}
Then there exists a polytope $P_{0}$ such that
\[
h(P_{0},\cdot)\mu^{tor}(P_{0},\cdot)=\sum_{k=1}^{N}\beta_{k}\delta_{v_{k}}(\cdot).
\]
\end{theo}

\begin{proof}

For $\delta_{1},\ldots,\delta_{N}\in {\R}$, sufficiently small $|t|>0$, define the polytope $P_{t}$ as

\begin{equation*}\label{}
P_{t}=\bigcap_{j=1}^{N}\{x: x\cdot v_{j}\leq h(P,v_{j})+t\delta_{j},\ j=1,\ldots,N\},
\end{equation*}
and satisfy
\begin{equation*}\label{}
\alpha(t)P_{t}=T(P_{t})^{-\frac{1}{n+2}}P_{t}.
\end{equation*}
Thus, from the part (i) of Lemma \ref{T1}, it is to see that $T(\alpha(t)P_{t})=1$, $\alpha(t)P_{t}\in \mathcal{P}_{N}(u_{1},\ldots,u_{N})$ and $\alpha(t)P_{t}\rightarrow P$ as $t\rightarrow 0$. Meanwhile, employ$~\eqref{extre4}$, let $\gamma(t)=\gamma_{p}(\alpha(t)P_{t})$ and
\begin{align}\label{pthmax6}
\Phi(t)&=\max_{\gamma\in \alpha(t)P_{t}}\sum_{k=1}^{N}\beta_{k}\log(\alpha(t)h(P_{t},v_{k})-\gamma\cdot v_{k})\notag\\
&=\sum_{k=1}^{N}\beta_{k}\log(\alpha(t)h(P_{t},v_{k})-\gamma(t)\cdot v_{k}).
\end{align}
By  ~\eqref{pthmax6}, Lemma \ref{INT}, and the fact is that $\gamma(t)$ is an interior point of $\alpha(t)P(t)$, then we get
\begin{equation*}\label{gam}
\frac{\partial \Phi(t)}{\partial \gamma({t})}=0.
\end{equation*}
Consequently, we have
\begin{equation}\label{gamva}
\sum_{k=1}^{N}\frac{\beta_{k}{v_{k,i}}}{\alpha(t)h(P_{t},v_{k})-\gamma(t)\cdot v_{k}}=0,
\end{equation}
for $i=1,\ldots,n$, $\gamma=(\gamma_{1},\ldots,\gamma_{n})^{T}$,  and $v_{k}=(v_{k,1},\ldots,v_{k,n})^{T}$.

By using ~\eqref{gamva}, let $t=0$ with $\alpha(0)=1, \gamma(0)=o$, then we get
\begin{equation}\label{t0equ}
\sum_{k=1}^{N}\frac{\beta_{k}v_{k,i}}{h(P,v_{k})}=0,\quad i=1,\ldots,n.
\end{equation}
Next, we reveal that $\gamma^{'}(t)\big|_{t=0}$ exists.

Let
\begin{equation*}\label{}
F_{i}(t,\gamma_{1},\ldots,\gamma_{n})=\sum_{k=1}^{N}\frac{\beta_{k}v_{k,i}}{\alpha(t)h(P_{t},v_{k})-(\gamma_{1}v_{k,1}+\ldots+\gamma_{n}v_{k,n})}
\end{equation*}
for $i=1,\ldots,n$.

Then, we have
\begin{equation*}\label{}
\frac{\partial F_{i}}{\partial \gamma_{j}}\Big|_{(0,\ldots,0)}=\sum_{k=1}^{N}\frac{\beta_{k}v_{k,i}v_{k,j}}{h(P,v_{k})^{2}}
\end{equation*}
for $j=1,\ldots, n$. Hence
\begin{equation*}\label{}
\left(\frac{\partial F}{\partial \gamma}\Big|_{(0,\ldots,0)}\right)_{n\times n}=\sum_{k=1}^{N}\frac{\beta_{k}}{h(P,u_{k})^{2}}v_{k}v_{k}^{T},
\end{equation*}
where $v_{k}v_{k}^{T}$ is an $n\times n$ matrix.

On the one hand, since $v_{1},\ldots,v_{N}$ are not concentrated on any closed hemisphere, for any $x\in{\rnnn}$, $x\neq \{0\}$, there exists $v_{i_{0}}\in \{{v_{1},\ldots,v_{N}}\}$ to meet $v_{i_{0}}\cdot x\neq 0$, and obtain
\begin{align*}\label{}
&x^{T}\left(\sum_{k=1}^{N}\frac{\beta_{k}}{h(P,v_{k})^{2}}v_{k}v_{k}^{T}\right)x\notag\\
&=\sum_{k=1}^{N}\frac{\beta_{k}(x\cdot v_{k})^{2}}{h(P,v_{k})^{2}}\notag\\
&\geq \frac{\beta_{i_{0}}(x\cdot v_{i_{0}})^{2}}{h(P,v_{i_{0}})^{2}}>0,
\end{align*}
which shows that $\frac{\partial F}{\partial \gamma}\big|(0,\ldots,0)$ is positively definite. From the inverse function theorem, we assert that $\gamma^{'}(0)=(\gamma_{1}^{'}(0),\ldots,\gamma_{n}^{'}(0))$ exists.

On the other hand, since $\Phi(0)$ attains the minimum value of $\Phi(t)$, by using \eqref{t0equ}, we obtain
\begin{align*}\label{}
0&=\frac{d\Phi(t)}{dt}\Bigg|_{t=0}\notag\\
&=\sum_{k=1}^{N}\beta_{k}h(P,v_{k})^{-1}\left[h(P,v_{k})\left(-\frac{1}{n+2}\right)\frac{d T(P_{t})}{dt}\Bigg|_{t=0}+\delta_{k}-\gamma^{'}(0)\cdot v_{k}\right]\notag\\
&=\sum_{k=1}^{N}\beta_{k}h(P,v_{k})^{-1}\left[-\frac{1}{n+2}h(P,v_{k})\left(\sum_{i=1}^{N}\delta_{i}\mu^{tor}(P,\{v_{i}\})\right)+\delta_{k}\right]-\gamma^{'}(0)\cdot\left[\sum_{k=1}^{N}\beta_{k}h(P,v_{k})^{-1}v_{k}\right]\notag\\
&=\sum_{k=1}^{N}\beta_{k}h(P,v_{k})^{-1}\left[-\frac{1}{n+2}h(P,v_{k})\left(\sum_{i=1}^{N}\delta_{i}\mu^{tor}(P,\{v_{i}\})\right)+\delta_{k}\right]\notag\\
&=\sum_{k=1}^{N}\delta_{k}\left[\beta_{k}h(P,v_{k})^{-1}-\frac{1}{n+2}\left(\sum_{i=1}^{N}\beta_{i}\right)\mu^{tor}(P,\{v_{k}\})\right].
\end{align*}
Since $\delta_{1},\ldots,\delta_{N}$ are arbitrary, we have
\begin{equation*}\label{}
\beta_{k}h(P,v_{k})^{-1}=\frac{1}{n+2}\left(\sum_{i=1}^{N}\beta_{i}\right)\mu^{tor}(P,\{v_{k}\})
\end{equation*}
for all $k=1,\ldots,N$. In view of the fact that $P$ is $n$-dimensional and $o\in Int(P)$, as a result, $h(P,v_{k})>0$, therefore,
\begin{equation}\label{Icon}
\beta_{k}=h(P,v_{k})\frac{1}{n+2}\left(\sum_{i=1}^{N}\beta_{i}\right)\mu^{tor}(P,\{v_{k}\}).
\end{equation}
Let $c=\frac{1}{n+2}\left(\sum_{i=1}^{N}\beta_{i}\right)$, apply ~\eqref{bordef} into ~\eqref{Icon}, we have
\begin{equation}\label{Icon2}
ch(P,\cdot)d\mu^{tor}(P,\cdot)=d\mu.
\end{equation}
Moreover, we have
\begin{equation*}\label{}
ch(P,\cdot)d\mu^{tor}(P,\cdot)=d\mu.
\end{equation*}
For $m>0$, $P\in \mathcal{P}(v_{1},\ldots,v_{N})$, by the part (a) of Lemma \ref{meau1}, we get
\begin{equation*}\label{}
d\mu^{tor}_{0}(mP,\cdot)=m^{n+2}h(P,\cdot)d\mu^{tor}(P,\cdot).
\end{equation*}
Then, let $P_{0}=\left[\frac{1}{n+2}(\sum_{i=1}^{N}\beta_{i})\right]^{\frac{1}{n+2}}P$, we get
\begin{equation*}\label{}
d\mu^{tor}_{0}(P_{0},\cdot)=d\mu,
\end{equation*}
which verifies that
\[
\mu^{tor}_{0}(P_{0},\cdot)=\sum_{k=1}^{N}\beta_{k}\delta_{v_{k}}(\cdot).
\]
This illustrates that $P_{0}$ is the desired polytope, which completes the proof.
\end{proof}

Applying Lemma \ref{mexist} and Lemma \ref{maint2}, we obtain the existence of solution to the discrete torsion log-Minkowski problem.
\begin{theo}\label{main21}
Let $\mu$ be a discrete measure on $\rnnn$ whose support set is not contained in any closed hemisphere and is in general position in dimension $n$. Then there exists a polytope $P$ containing the origin in its interior such that
\[
G^{tor}(P,\cdot)=\mu.
\]

\end{theo}

\section{The general torsion log-Minkowski problem}
\label{Sec5}
In this section, our aim is to deal with the general torsion log-Minkowski problem by virtue of the approximation technique.

Given a finite Borel measure $\mu$ on the unit sphere ${\sn}$, not concentrated on any closed hemisphere, we first construct a sequence of discrete measures whose support sets are in general position such that sequence of discrete measures converges to $\mu$ weakly.

For each positive integer $j$, divide $\sn$ into a finite disjoint union
\[
\sn=\bigcup^{N_{j}}_{i=1}U_{i,j}
\]
for $N_{j}>0$, which satisfies that the diameter of $U_{i,j}$ is less than $\frac{1}{j}$, and $U_{i,j}$ contains nonempty interior. Then, we can choose $v_{i,j}\in U_{i,j}$ such that $v_{1,j},\ldots, v_{N_{j},j}$ are in general position. As $j\rightarrow \infty$, we know that the vector $v_{1,j},\ldots, v_{N_{j},j}$ can not contained in any closed hemisphere.

We are in a position to define the discrete measure $\mu_{j}$ on $\sn$ as
\[
\mu_{j}=\sum^{N_{j}}_{i=1}\left( \mu(U_{i,j})+\frac{1}{N^{2}_{j}} \right)\delta_{v_{i,j}},
\]
and define
\begin{equation}\label{UJ}
\bar{\mu}_{j}=\frac{|\mu|}{|\mu_{j}|}\mu_{j}.
\end{equation}

 From the above statement, we see that $\bar{\mu}_{j}$ is a discrete measure on $\sn$ satisfying the conditions in Theorem \ref{maint2}, and $\bar{\mu}_{j} \rightharpoonup \mu$ weakly. Consequently, by Theorem \ref{maint2}, there is a polytope $\tilde{P}_{j}$  such that
\begin{equation}\label{ajlim}
h(\tilde{P}_{j},\cdot)\mu^{tor}(\tilde{P}_{j},\cdot)=\bar{\mu}_{j},
\end{equation}
 and $\tilde{P}_{j}$ is a rescaled of $P_{j}$, i.e.,
  \begin{equation}\label{PPJ}
\tilde{P}_{j}=\left(\frac{1}{n+2}|\bar{\mu}_{j}| \right)^{\frac{1}{n+2}}P_{j},
\end{equation}
  where $P_{j}$ satisfies $P_{j}\in P(v_{1,j},\ldots,v_{N_{j},j})$, $\gamma_{j}(P_{j})=o$, $T(P_{j})=1$, and solves
\begin{equation}\label{LP}
\Phi_{P_{j}}(o)=\inf\left\{\max_{\gamma_{j}\in Int (Q_{j})}\Phi_{Q_{j}}(\gamma_{j}):Q_{j}\in \mathcal{P}(v_{1,j},\ldots,v_{N_{j},j}),T(Q_{j})=1\right\}.
\end{equation}

Now, we introduce the following key lemma, which shall be used in the below.
\begin{lem}\label{N1}
Let $v_{1,j},\ldots,v_{N_{j},j}\in \sn$ be as given above. Set
\begin{equation}\label{SJ}
\mathcal{S}_{j}=\sum^{N_{j}}_{j=1}\{x\in \rnnn: \ x\cdot v_{i,m}\leq 1 \}.
\end{equation}
Then, for sufficiently large $j$, we get
\[
B^{n}\subset \mathcal{S}_{j} \subset 3  B^{n}.
\]
\end{lem}
\begin{proof}
 As indicated before, since $\{U_{i,j}\}_{j}$ is a partition of $\sn$, then for each $u\in \sn$, there exists $i_{j}$ such that $u\in U_{i_{j},j}$. Since $diam(U_{i,j})<\frac{1}{j}$. So, we can choose $N_{0}>0$ such that for each $j>N_{0}$,
 \[
 u\cdot v_{i_{j},j}>\frac{1}{3}.
 \]
Due to $\rho(\mathcal{S}_{j},u)u\in \mathcal{S}_{j}$, there is
\[
\rho(\mathcal{S}_{j},u)/3<\rho(\mathcal{S}_{j},u)u\cdot v_{i_{j},j}\leq 1.
\]
Consequently, $\rho(\mathcal{S}_{j},\cdot)< 3$ for each $j>N_{0}$, which gives the desired result.
\end{proof}
Without loss of generality, for $\gamma \in \Omega$, we write
\begin{equation}\label{DD2}
\Phi_{\Omega,\mu}(\gamma)=\int_{\sn}\log h(\Omega-\gamma,\cdot)d\mu.
\end{equation}
Note that, when $\mu$ is a discrete measure, and $\{v_{1},\ldots, v_{N}\}$ is the support of $\mu$, then \eqref{DD2} is precisely \eqref{bordef} given by in Sec. \ref{Sec3}.

By means of lemma \ref{N1}, we have the following result.

\begin{lem}\label{PES}
Let $\tilde{P}_{j}$ be as given in \eqref{ajlim}, and $\gamma_{j}(P_{j})$ is the minimizer to \eqref{LP} with $\gamma_{j}(P_{j})=o$. If $|u|$ is positive, then there exists a constant $c_{0}>0$, independent of $j$, such that
\begin{equation}\label{UYT}
\Phi_{\tilde{P}_{j},\bar{\mu}_{j}}(o)<c_{0}.
\end{equation}
\end{lem}
\begin{proof}
Set $\mathcal{S}_{j}$ as defined in \eqref{SJ}. Using lemma \ref{N1}, we know that, for $r>0$ and for sufficiently large $j$, one see $rB^{n}\subset r \mathcal{S}_{j}\subset 3r B^{n}$. By virtue of the homogeneity of $T$, there exists $r_{0}(j)>0$ such that
\[
T(r_{0}(j)\mathcal{S}_{j})=1.
\]
Due to $r B\subset r \mathcal{S}_{j}$, there is
\[
r_{0}(j)^{n+2}T(B^{n})=T(r_{0}(j)B^{n})\leq T(r_{0}(j)\mathcal{S}_{j})=1
\]
Thus, $r_{0}(j)\leq r_{0}$ for some constant $r_{0}$ independent of $j$, and
\begin{equation}
\begin{split}
\label{PPJ2}
\Phi_{P_{j},\bar{\mu}_{j}}(o)&\leq \int_{\sn}\log h(r_{0}(j)\mathcal{S}_{j}-\gamma_{j} (r_{0}(j)\mathcal{S}_{j}),\cdot)d \bar{\mu}_{j}\\
&\leq \int_{\sn} \log h(6r_{0}B^{n},\cdot)d\bar{\mu}_{j}=\log(6r_{0})|\bar{\mu}_{j}|=\log(6r_{0})|\mu|.
\end{split}
\end{equation}
Hence, the desired bound \eqref{UYT} follows from $\eqref{PPJ}$ and \eqref{PPJ2}.
\end{proof}

As elaborated before, given a Borel measure $\mu$ on $\sn$, $\mu$ satisfies the \emph{subspace mass inequality} provided
\begin{equation}\label{space}
\frac{\mu(\xi_{k}\cap \sn)}{|\mu|}< \frac{k}{n}
\end{equation}
for each $\xi_{k}\in G_{n,k}$, and for each $k=1,\ldots,n-1$.
\begin{theo}\cite{LU23}
If $\mu$ is an even, finite, non-zero Borel measure on $\sn$ satisfying the following subspace mass inequality \eqref{space}, then there exists an origin-symmetric convex body $\Omega\subset \rnnn$ such that
\[
G^{tor}(\Omega,\cdot)=\mu.
\]
\end{theo}

Our aim is to derive the above theorem without symmetric assumptions via an approximation argument. For this purpose, we need to the following preparation.
\begin{lem}\label{UY}
Let $\tilde{P}_{j}$ be as given in \eqref{ajlim}. Then there exists $c_{0}>0$ such that \[
T(\tilde{P}_{j})>c_{0}\]
for every $j$.
\end{lem}
\begin{proof}
For a convex body $\Omega$ in $\rnnn$, with the aid of the definition of $G^{tor}(\cdot)$ and $T(\cdot)$, one see
\begin{equation}\label{TTR}
T(\tilde{P}_{j})=|G^{tor}(\tilde{P}_{j},\sn)|=\frac{1}{n+2}|\bar{\mu}_{j}|=\frac{1}{n+2}|\mu|:=c_{0}>0.
\end{equation}
Thus, we get the desired result.
\end{proof}

Next, we shall prove that $P_{j}$ is uniformly bounded when $\mu$ (not necessarily even) satisfies the subspace mass inequality \eqref{space}. For convenience, we write
\[
\chi_{k}=\frac{k}{n}.
\]
For each $\omega\subset \sn$ and $\eta >0$, we define
\[
\Re_{\eta}(\omega)=\{v\in \sn: |v-u|< \eta, \ {\rm for \ some \ } u\in \omega\}.
\]

Motivated by \cite[Lemma 4.1]{CL19} or \cite[Lemma 5.10]{GXZ23}, we get a slightly stronger subspace mass inequality for sufficiently large $j$.
\begin{lem}\label{IU}
Let $\mu$ be a finite Borel measure on $\sn$ and $\bar{\mu}_{j}$ be established in \eqref{UJ}, if $\mu$ satisfies the subspace mass inequality \eqref{space}, then there exist $\tilde{\chi}_{k}\in (0,\chi_{k})$, $N_{0}>0$, and $\eta_{0}\in (0,1)$ such that for all $j>N_{0}$,
\begin{equation}\label{kkk}
\frac{\bar{\mu}_{j}(\Re_{\eta_{0}}(\xi_{k}\cap \sn))}{|\mu|}<\tilde{\chi}_{k}
\end{equation}
for each $k$-dimensional subspace $\xi_{k}\subset \rnnn$ and $k=1,\ldots,n-1$.
\end{lem}
\begin{proof}
We take by contradiction argument. Suppose that \eqref{kkk} does not hold, then there exists a sequence of subspaces $\xi^{i}_{k}$ of $\xi_{k}$ such that $j_{i}\rightarrow \infty$, $\eta_{i}\rightarrow 0$, $\chi^{i}_{k}\rightarrow \chi_{k}$ as $i\rightarrow \infty$, and
\begin{equation}\label{kkk1}
\frac{\bar{\mu}_{j_{i}}(\Re_{\eta_{i}}(\xi^{i}_{k}\cap \sn))}{|\mu|}\geq\chi^{i}_{k}
\end{equation}
Now, let $\{e_{1,i},\ldots, e_{k,i}\}$ be an orthonormal basis of $\xi^{i}_{k}$. By choosing suitable subsequences, we may suppose that $e_{1,i}\rightarrow e_{1},\ldots, e_{k,i}\rightarrow e_{k}$ as $i\rightarrow \infty$. Let $\xi_{k}={\rm span} \{{e_{1},\ldots, e_{k}}\}$. Since $\eta_{i}\rightarrow 0$ as $i\rightarrow \infty$, given a small positive $\eta$,  we get
\begin{equation*}
\begin{split}
\label{qw}
\frac{\bar{\mu}_{j_{i}}(\overline{\Re_{\eta}(\xi_{k}\cap \sn)})}{|\mu|}\geq\frac{\bar{\mu}_{j_{i}}(\Re_{\eta_{i}}(\xi^{i}_{k}\cap \sn))}{|\mu|}\geq\chi^{i}_{k}
\end{split}
\end{equation*}
for large enough $i$. In view of the fact that $\bar{\mu}_{j_{i}}$ converges weakly to $\mu$,  $\overline{\Re_{\eta}(\xi_{k}\cap \sn)}$ is compact, and $\chi^{i}_{k}\rightarrow \chi_{k}$, as $i\rightarrow \infty$, then we obtain
\[
\frac{\mu(\overline{\Re_{\eta}(\xi_{k}\cap \sn)})}{|\mu|}\geq \chi_{k}.
\]
Since $\eta$ is arbitrary, let $\eta\rightarrow 0$, there is
\[
\frac{\mu(\xi_{k}\cap \sn)}{|\mu|}\geq \chi_{k}.
\]
This contradicts to \eqref{space}. Hence, we complete the proof.
\end{proof}

Based on lemma \ref{IU}, we will estimate the functional $\Phi_{E_{j},\bar{\mu}_{j}}(o)$, where $E_{j}$ is a sequence of centered ellipsoids. The key technique is to use an appropriate spherical partition. The general argument was introduced by \cite{BLYZ12}.

Let $e_{1},\ldots, e_{n}$ be an orthonormal basis in $\rnnn$. For each
$\delta\in (0,\frac{1}{\sqrt{n}})$, define the partion $\{A_{1,\delta},\ldots,A_{n,\delta}\}$ of $\sn$, with respect to $e_{1},\ldots,e_{n}$ by
\begin{equation}\label{ER1}
A_{k,\delta}=\{ v\in \sn: |\nu\cdot e_{k}|\geq \delta \ {\rm and} \ |v \cdot e_{j}|\leq \delta \ {\rm for \ all }\ j>k\},\ k=1,\ldots,n.
\end{equation}

Let
\[
\xi_{k}={\rm span}\{e_{1},\ldots,e_{k}\},\ k=1,\ldots,n,
\]
and $\xi_{0}=\{0\}$. It was shown in \cite{BLYZ12} that for any non-zero finite Borel measure $\mu$ on $\sn$,
\begin{equation*}\label{STE}
\lim_{\delta\rightarrow 0^{+}}\mu(A_{k,\delta})=\mu((\xi_{k}\backslash \xi_{k-1})\cap \sn),
\end{equation*}
and therefore,
\begin{equation*}\label{ST2}
\lim_{\delta\rightarrow 0^{+}}(\mu(A_{1,\delta})+\ldots+\mu(A_{k,\delta}))=\mu(\xi_{k}\cap \sn).
\end{equation*}
The following lemma from \cite{BLYZ19} shall be needed.
\begin{lem}\cite[Lemma 4.1]{BLYZ19}\label{BL}
Suppose $\lambda_{1},\ldots, \lambda_{m}\in [0,1]$ are such that
\[
\lambda_{1}+\ldots+\lambda_{m}=1.
\]
Suppose further that $a_{1},\ldots, a_{m}\in \R$ are such that
\[
a_{1}\leq a_{2}\leq \ldots \leq a_{m}.
\]
Assume that there exists $\sigma_{0},\sigma_{1},\ldots,\sigma_{m}\in [0,\infty)$, with $\sigma_{0}=0$, and $\sigma_{m}=1$ such that
\[
\lambda_{1}+\ldots+\lambda_{k}\leq \sigma_{k}, \quad {\rm for}\ k=1,\ldots,m.
\]
Then
\[
\sum^{m}_{k=1}\lambda_{k}a_{k}\geq \sum^{m}_{k=1}(\sigma_{k}-\sigma_{k-1})a_{k}.
\]

\end{lem}

\begin{lem}\label{EE}
Suppose $\varepsilon_{0}>0$. Suppose further that $(e_{1,j},\ldots,e_{n,j})$, where $j=1,2,\ldots,$ is  a sequence of ordered orthonormal bases of $\rnnn$ converging to the ordered orthonormal basis $(e_{1},...,e_{n})$, while $(a_{1,j},\ldots,a_{n,j})$ is a sequence of $n$-tuples satisfying
\[
0<a_{1,l}\leq a_{2,l}\leq \ldots\leq a_{n,j} \ and \ a_{n,j}\geq \varepsilon_{0}.
\]
For each $j=1,2,\ldots$, let
\[
E_{j}=\left\{x\in \rnnn:\frac{(x\cdot e_{1,j})^{2}}{a^{2}_{1,j}}+\ldots+\frac{(x\cdot e_{n,j})^{2}}{a^{2}_{n,j}}\leq 1\right\}
\]
denote the ellipsoids generated by the $(e_{1,j},\ldots,e_{n,j})$ and $(a_{1,l},\ldots,a_{n,l})$. Let $\mu$ be a nonzero finite Borel measure on $\sn$ satisfying the subspace mass inequality \eqref{space}. Then there exists $\delta_{0}, t_{0}\in (0,1)$ such that for each $j>N_{0}$, we have
\begin{equation}
\begin{split}
\label{EP}
\frac{1}{|\bar{\mu}_{j}|}\Phi_{E_{j},\bar{ \mu}_{j} }(o)\geq\log\left(\frac{\delta_{0}}{2}\right)+t_{0}\log a_{n,j}+\frac{1}{n}(1-t_{0})\log |E_{j}|-\frac{1}{n}(1-t_{0})\log\omega_{n}.
\end{split}
\end{equation}
\end{lem}

\begin{proof}
For each $\delta\in (0,1/\sqrt{n})$, let $\{A_{1,\delta},\ldots, A_{n,\delta}\}$ be the partition of $\sn$ as in \eqref{ER1}.

Since $\mu$ satisfies the subspace mass inequality $\eqref{space}$, by lemma \ref{IU}, we know that there exists $N_{0}>0$, $\eta_{0}\in (0,1)$, and $\tilde{\chi}_{k}\in (0,\chi_{k})$ such that for all $j>N_{0}$, \eqref{kkk} holds for each $k$-dimensional proper subspace $\xi_{k}\subset \rnnn$. Let $t_{0}$ be sufficiently small so that
\[
(1-t_{0})\chi_{k}> \tilde{\chi}_{k}.
\]
Thus, for all $j>N_{0}$, there is
\begin{equation}\label{II}
\frac{\bar{\mu}_{j}(\Re_{\eta_{0}}(\xi_{k}\cap \sn))}{|\mu|}< (1-t_{0})\chi_{k}
\end{equation}
for each $k$-dimensional subspace $\xi_{k}\subset \rnnn$ and $k=1,\ldots,n-1$. Particularly, we let $\xi_{k}=span\{e_{1},\ldots,e_{k}\}$,

Note that for sufficiently small $\delta_{0}\in (0,1)$, one see
\[
\bigcup^{k}_{i=1}A_{i,\delta_{0}}\subset \Re_{\eta_{0}}(\xi_{k}\cap \sn).
\]
Now,  let
\[
\lambda_{i,\delta_{0}}=\frac{\bar{\mu}_{j}(A_{i,\delta_{0}})}{|\bar{\mu}_{j}|}.
\]
Note that
\[
\lambda_{1,\delta_{0}}+\ldots+\lambda_{n,\delta_{0}}=1.
\]
By virtue of \eqref{II}, we obtain
\[
\lambda_{1,\delta_{0}}+\ldots+\lambda_{k,\delta_{0}}=\frac{\sum^{k}_{i=1}\bar{\mu}_{j}(A_{i,\delta_{0}})}{|\mu|}<(1-t_{0})\chi_{k}
\]
for each $k=1,\ldots,n-1$.
Let $\sigma_{n}=1$ and $\sigma_{0}=0$ and $\sigma_{k}=(1-t_{0})\chi_{k}=(1-t_{0})\frac{k}{n}$ for $k=1,\ldots,n-1$. Then there are
\begin{equation}\label{W1}
\sigma_{1}-\sigma_{0}=\frac{1}{n}(1-t_{0}),
\end{equation}
\begin{equation}\label{W2}
\sigma_{k}-\sigma_{k-1}=\frac{1}{n}(1-t_{0}), \ k=2,\ldots,n-1,
\end{equation}
and
\begin{equation}\label{W3}
\sigma_{n}-\sigma_{n-1}=\frac{1}{n}(1-t_{0})+t_{0}.
\end{equation}
On the other hand, since $\lim_{j\rightarrow \infty}e_{k,j}=e_{k}$ for each $k=1,\ldots,n$, we may choose $j_{0}>0$, so that $|e_{k,j}-e_{k}|< \delta_{0}/2$ for each $j>j_{0}$ and each $j=1,\ldots,n$. Since $(e_{1,j},\ldots,e_{n,j})$ is orthonormal, by the definition of $E_{j}$ that $\pm a_{k,j}e_{k,j}\in E_{j}$, for $v\in A_{k,\delta_{0}}$, we have
\begin{equation}\label{W4}
h(E_{j},v)\geq a_{k,j}|e_{k,j}\cdot v|\geq a_{k,j}(|v\cdot e_{k}|-|e_{k,j}-e_{k}|)\geq a_{k,j}\frac{\delta_{0}}{2}.
\end{equation}
In view of the fact that $\{A_{k,\delta_{0}}\}^{n}_{k=1}$ is a partition of $\sn$. Together with \eqref{W4} and lemma \ref{BL}, for $j>j_{0}$, we have
\begin{equation}
\begin{split}
\label{TOT}
\frac{1}{|\bar{\mu}_{j}|}\Phi_{E_{j},\bar{ \mu}_{j} }(o)&=\frac{1}{|\bar{\mu}_{j}|}\int_{\sn}\log h(E_{j},\cdot)d\bar{\mu}_{j}\\
&\geq\sum^{n}_{k=1}\frac{\bar{\mu}_{j} (A_{k,\delta_{0}})}{|\bar{\mu}_{j}|}\log \left(a_{k,j}\frac{\delta_{0}}{2}\right)\\
&=\log \left(\frac{\delta_{0}}{2}\right)+\sum^{n}_{k=1}\lambda_{k,\delta_{0}}\log a_{k,j}\\
&\geq\log \left(\frac{\delta_{0}}{2}\right)+\sum^{n}_{k=1}(\sigma_{k}-\sigma_{k-1})\log a_{k,j}.
\end{split}
\end{equation}
Now, applying \eqref{W1}, \eqref{W2} and \eqref{W3} into \eqref{TOT}, we have
\begin{equation}
\begin{split}
\label{EP11}
\frac{1}{|\bar{\mu}_{j}|}\Phi_{E_{j},\bar{ \mu}_{j} }(o)&\geq \log \left(\frac{\delta_{0}}{2}\right)+\frac{1}{n}(1-t_{0})\sum^{n}_{k=1}\log a_{k,j}+t_{0}\log a_{n,j}.
\end{split}
\end{equation}
Then \eqref{EP} holds by substituting $|E_{j}|=\omega_{n}a_{1,j}a_{2,j}\ldots a_{n,j}$ into \eqref{EP11}.
\end{proof}
Now, we prove that $P_{j}$ is uniformly bounded.

\begin{lem}\label{PKL}
Let $\mu$ be a nonzero finite Borel measure on $\sn$ and $\bar{\mu}_{j}$ be as given in \eqref{UJ}. Let $\tilde{P}_{j}$ be as constructed in \eqref{ajlim}. If $\mu$ satisfies the subspace mass inequality $\eqref{space}$, the $\tilde{P}_{j}$ is uniformly bounded.
\end{lem}
\begin{proof}
We take by contradiction and assume that $\tilde{P}_{j}$ is not uniformly bounded.

Let $E_{j}$ be the John ellipsoid of $\tilde{P}_{j}$, then
\begin{equation}\label{EJJ}
E_{j}\subset \tilde{P}_{j}\subset n(E_{j}-o_{j})+o_{j},
\end{equation}
where the ellipsoid $E_{j}$ centered at $o_{j}\in Int (\tilde{P}_{j})$ is given by
\[
E_{j}=\left\{x\in \rnnn:\frac{|(x-o_{j})\cdot e_{1,j}|^{2}}{a^{2}_{1,j}}+\ldots+\frac{|(x-o_{j})\cdot e_{n,j}|^{2}}{a^{2}_{n,j}}\leq 1\right\}
\]
for a sequence of ordered orthonormal bases $(e_{1,j},\ldots,e_{n,j})$ of $\rnnn$, with
$0<a_{1,l}\leq \ldots \leq a_{n,l}$. Since $\tilde{P}_{j}$ is not uniformly bounded, by choosing a subsequence, we may assume $a_{n,j}\rightarrow \infty$ and $a_{n,j}\geq 1$. By means of the compactness of $\sn$, we can take a subsequence and assume that $\{e_{1,j},\ldots, e_{n,j}\}$ converges to $\{e_{1},\ldots, e_{n}\}$ as an orthonormal basis in $\rnnn$. Using Lemma \ref{JKT},  Lemma \ref{EE}, \eqref{EJJ}, there exists $\delta_{0}$, $t_{0}$, $c_{n}$ and $N_{0}>0$ such that
\begin{equation}
\begin{split}
\label{QWE}
\frac{1}{|\bar{\mu}_{j}|}\Phi_{\tilde{P}_{j},\bar{ \mu}_{j} }(o_{j})&\geq \frac{1}{|\bar{\mu}_{j}|}\Phi_{E_{j},\bar{ \mu}_{j} }(o_{j})\\
&= \frac{1}{|\bar{\mu}_{j}|}\Phi_{E_{j}-o_{j},\bar{ \mu}_{j} }(o)\\
&\geq\log\left(\frac{\delta_{0}}{2}\right)+t_{0}\log a_{n,j}+\frac{1}{n}(1-t_{0})\log |E_{j}|-\frac{1}{n}(1-t_{0})\log\omega_{n}\\
&\geq\log\left(\frac{\delta_{0}}{2}\right)+t_{0}\log a_{n,j}+\frac{1}{n+2}(1-t_{0})\log |T(E_{j})|+c(n,t_{0}),
\end{split}
\end{equation}
where $c(n,t_{0})$ is not necessarily positive.

To proceed further, with the aid of the homogeneity, monotonicity and translation invariance of $T$, \eqref{IS} and \eqref{EJJ}, we obtain
\begin{equation}
\begin{split}
\label{JIU}
T(E_{j})=T(E_{j}-o_{j})&=n^{-\frac{1}{n+2}}T(n(E_{j}-o_{j})+o_{j})\\
&\geq n^{-\frac{1}{n+2}} T(\tilde{P}_{j})=n^{-\frac{1}{n+2}}\frac{1}{n+2}|\bar{\mu}_{j}|.
\end{split}
\end{equation}

Recall that  $\gamma_{j}(P_{j})$ is the minimizer to \eqref{LP} with $\gamma_{j}(P_{j})=o$, by applying \eqref{PPJ}, we know that $\gamma_{j}(\tilde{P}_{j})=o$, this combines with $|\bar{\mu}_{j}|=|\mu|$, \eqref{JIU}, and that $a_{n,j}\rightarrow \infty$, from \eqref{QWE}, we have
\begin{equation*}\label{}
\Phi_{\tilde{P}_{j},\bar{\mu}_{j}}(o)\geq\Phi_{\tilde{P}_{j},\bar{\mu}_{j}}(o_{j})\rightarrow \infty, \  {\rm as} \ j\rightarrow \infty,
\end{equation*}
which contradicts to Lemma \ref{PES}. The proof is completed.
\end{proof}

\begin{lem}\label{POI}
If $\tilde{P}_{j}$ in \eqref{ajlim} are uniformly bounded and $T(\tilde{P}_{j})>c_{0}$ for some constant $c_{0}>0$, then there exists a convex body $\Omega\in \koo$ with $o\in \Omega$ such that
\begin{equation}\label{HG}
h(\Omega,\cdot)\mu^{tor}(\Omega,\cdot)=\mu.
\end{equation}
\end{lem}
\begin{proof}
Since $\tilde{P}_{j}$ is uniformly bounded, with the aid of the Blaschke selection theorem, there exists a subsequence $\tilde{P}_{j_{i}}$ such that $\tilde{P}_{j_{i}}\rightarrow \Omega$ for some compact convex set $\Omega$ containing the origin. By virtue of the continuity of $T$ and the fact that $T(\tilde{P}_{j})>c_{0}>0$, one see $T(\Omega)>0$. From Lemma \ref{JKT}, we know that $|\Omega|\geq c_{1}>0$, which implies that $\Omega$ has nonempty interior. Finally, by taking the limit of \eqref{ajlim} on both sides and employing the weak convergence of $\mu^{tor}(\cdot)$ and the uniform continuity of the support function, we conclude that \eqref{HG} holds.
\end{proof}

From Lemma \ref{UY}, Lemma \ref{PKL} and Lemma \ref{POI}, we shall give a solution to the general torsion log-Minkowski problem described as below.

\begin{theo}\label{IUY}
If $\mu$ is a finite, non-zero Borel measure on $\sn$ satisfying the subspace mass inequality \eqref{space}. Then there exists a convex body $\Omega\subset \rnnn$  with $o\in \Omega$ such that
\[
G^{tor}(\Omega,\cdot)=\mu.
\]

\end{theo}

\end{document}